\documentclass[12pt,leqno]{amsart}
\usepackage{amssymb,amsthm,amsmath,latexsym}
\newtheorem{theorem}{\sc Theorem}[section]
\newtheorem{thm}[theorem]{\sc Theorem}
\newtheorem{lem}[theorem]{\sc Lemma}
\newtheorem{prop}[theorem]{\sc Proposition}

\newtheorem{rem}[theorem]{\sc Remark}
\newtheorem{ex}[theorem]{\sc Example}

 \newtheorem*{thmA}{Theorem A}
 \newtheorem*{thmB}{Theorem B}
 
\title{On finite groups with few automorphism orbits}
\author{Raimundo Bastos and Alex Carrazedo Dantas}
\subjclass[2010]{20D05; 20D45.}
\keywords{Finite groups; automorphism groups}

\begin{document}
\maketitle

\begin{abstract}
Denote by $\omega(G)$ the number of orbits of the action of $Aut(G)$ on the finite group $G$. We prove that if $G$ is a finite nonsolvable group in which $\omega(G) \leqslant 5$, then $G$ is isomorphic to one of the groups $A_5,A_6,PSL(2,7)$ or $PSL(2,8)$. We also consider the case when $\omega(G) = 6$ and show that if $G$ is a nonsolvable finite group with $\omega(G) = 6$, then either $G \simeq PSL(3,4)$ or there exists a characteristic elementary abelian $2$-subgroup $N$ of $G$ such that $G/N \simeq A_5$.
\end{abstract}

\maketitle

\section{Introduction}

The groups considered in the following are finite. The problem of the classification of the groups with a prescribed number of {\emph conjugacy classes} was suggested in \cite{B}. For more details for this problem we refer the reader to \cite{VS}. In this paper we consider an other related invariant. Denote by $\omega(G)$ the number of orbits of the action of $Aut(G)$ on the finite group $G$. If $\omega(G) = n$, then we say that $G$ has $n$ automorphism orbits. The trivial group is the only group with $\omega(G) = 1$. It is clear that $\omega(G) = 2$ if and only if $G$ is an elementary abelian $p$-group, for some prime number $p$ \cite[3.13]{D}. In \cite{LD}, Laffey and MacHale give the following results:

\begin{itemize}
\item[(i)]  Let $G$ be a finite group which is not of prime-power order. If $w(G) = 3$, then $|G| = p^nq$ and $G$ has a normal elementary abelian Sylow $p$-subgroup $P$, for some primes $p$, $q$, and for some integer $n > 1$. Furthermore, $p$ is a primitive root mod $q$.
\item[(ii)] If $w(G)\leqslant 4$ in a group $G$, then either $G$ is solvable or $G$ is isomorphic to $A_5$.
\end{itemize}

Stroppel in \cite[Theorem 4.5]{S} has shown that if $G$ is a nonabelian simple group with $\omega(G) \leqslant 5$, then $G$ is isomorphic to one of the groups $A_5$, $A_6$, $PSL(2,7)$ or $PSL(2,8)$. In the same work he suggested the following problem: \\

{\noindent}{\bf Problem.}{
(Stroppel \cite[Problem 9.9]{S}) Determine the finite nonsolvable groups $G$ in which $\omega(G) \leqslant 6$.} \\

In answer to Stroppel's question, we give a complete classification for the case $\omega(G)\leq 5$ in Theorem A and provide a characterization of $G$ when $\omega(G)= 6$ in Theorem B. Precisely:

\begin{thmA}\label{th.A}
Let $G$ be a nonsolvable group in which $\omega(G) \leqslant 5$. Then $G$ is isomorphic to one of the groups $A_5,A_6,PSL(2,7)$ or $PSL(2,8)$. 
\end{thmA}

Using GAP, we obtained an example of a nonsolvable and non simple group $G$ in which $\omega(G) = 6$ and $|G|=960$. Moreover, there exists a characteristic subgroup $N$ of $G$ such that $G/N \simeq A_5$, where $N$ is an elementary abelian $2$-subgroup. Actually, we will prove that this is the case for any nonsolvable non simple group with $6$ automorfism orbits.

\begin{thmB} \label{th.B}
Let $G$ be a nonsolvable group in which $\omega(G) = 6$. Then one of the following holds:
\end{thmB}

\begin{itemize}
 \item[(i)] $G \simeq PSL(3,4)$; 
 \item[(ii)] There exists a characteristic elementary abelian $2$-subgroup $N$ of $G$ such that $G/N \simeq A_5$.
\end{itemize}

According to Landau's result \cite[Theorem 4.31]{Rose}, for every positive integer $n$, there are only finitely many groups with exactly $n$ conjugacy classes. It is easy to see that no exists similar result for automorphism orbits. Nevertheless, using the classification of finite simple groups, Kohl has been able to prove that for every positive integer $n$ there are only finitely many nonabelian simple groups with exactly $n$ automorphism orbits \cite[Theorem 2.1]{Kohl}.  This suggests the following question. \\

{\noindent}{\it Are there only finitely many nonsolvable groups with $6$ automorphism orbits?} 

\section{Preliminary results}

A group $G$ is called AT-group if all elements of the same order are conjugate in the automorphism groups. The following result is a straightforward corollary of \cite[Theorem 3.1]{Z}.

\begin{lem} \label{Z.th}
Let $G$ be a nonsolvable AT-group in which $\omega(G) \leqslant 6$. Then $G$ is simple. Moreover, $G$ is isomorphic to one of the groups $A_5$, $A_6$, $PSL(2,7)$, $PSL(2,8)$ or $PSL(3,4)$.
\end{lem}

The spectrum of a group is the set of orders of its elements. Let us denote by $spec(G)$ the spectrum of the group $G$.

\begin{rem} \label{rem.spec} The maximal subgroups of $A_5$, $A_6$, $PSL(2,7)$, $PSL(2,8)$ and $PSL(3,4)$ are well know (see for instance \cite{Atlas}). Then  
\begin{itemize}
\item[(i)] $spec(A_5) = \{1,2,3,5\}$;
\item[(ii)] $spec(A_6) = \{1,2,3,4,5\}$;
\item[(iii)] $spec(PSL(2,7)) = \{1,2,3,4,7\}$;
\item[(iv)] $spec(PSL(2,8)) = \{1,2,3,7,9\}$;
\item[(v)] $spec(PSL(3,4)) = \{1,2,3,4,5,7\}$
\end{itemize}
\end{rem}

For a group $G$ we denote by $\pi(G)$ the set of prime divisors of the orders of the elements of $G$. 

Recall that a group $G$ is a characteristically simple group if $G$ has no proper nontrivial characteristic subgroups.

\begin{lem} \label{ch.simple}
Let $G$ be a nonabelian group. If $G$ is a characteristically simple group in which $\omega(G) \leqslant 6$, then $G$ is simple.
\end{lem}

\begin{proof}
Suppose that $G$ is not simple. By \cite[Theorem 1.5]{G}, there exist a nonabelian simple subgroup $H$ and an integer $k \geqslant 2$ such that
$$
G = \underbrace{H \times \ldots \times H}_{k \ times}.
$$
By Burnside's Theorem \cite[p. 131]{G}, $\pi(G) =\{p_1,\ldots, p_s\}$, where $s \geqslant 3$. Then, there are elements in $G$ of order $p_ip_j$, where $i,j \in \{1,\ldots,s\}$ and $i \neq j$. Thus, $\omega(G) \geqslant 7$.
\end{proof}

\begin{lem} \label{simple-lem}
Let $G$ be a nonsolvable group and $N$ a characteristic subgroup of $G$. Assume that $|\pi(G)| = 4$ and $N$ is isomorphic to one of the groups $A_5, A_6, PSL(2,7)$ or $PSL(2,8)$. Then $\omega(G) \geqslant 8$. 
\end{lem}

\begin{proof}
Let $P$ be a Sylow $p$-subgroup of $G$, where $p \not\in \pi(N)$. Set $M = NP$. Since $p$ and $|Aut(N)|$ have coprime orders, we conclude that $M = N \times P$. Arguing as in the proof of Lemma \ref{ch.simple} we deduce that $\omega(G) \geqslant 8$. 
\end{proof}

\begin{rem} \label{prop.spec}
(Stroppel, \cite[Lemma 1.2]{S}) Let $G$ be a nontrivial group and $K$ a characteristic subgroup of $G$. Then
$$
\omega(G) \geqslant \omega(K)+ \omega(G/K) - 1.
$$
\end{rem}

\begin{lem} \label{lem.spec}
Let $G$ be a nonsolvable group in which  $|spec(G)|\geqslant 6$. Then either $G$ is simple or $\omega(G) \geqslant 7$.
\end{lem}

\begin{proof}
Assume that $\omega(G)= 6$. Then, $G$ is AT-group. By Lemma \ref{Z.th}, $G$ is simple. Moreover, $G \simeq PSL(3,4)$.       
\end{proof}

\begin{prop} \label{A5-prop}
Let $G$ be a group and $N$ a proper characteristic subgroup of $G$. If $N$ is isomorphic to one of the following groups $A_5,A_6,PSL(2,7)$ or $PSL(2,8)$, then $\omega(G) \geqslant 7$.
\end{prop}

\begin{proof}  
By Lemma \ref{simple-lem}, there is no loss of generality in assuming $\pi(G) = \pi(N)$. In particular, there is a subgroup $M$ in $G$ such that $|M| = p|N|$, for some prime $p \in \pi(G)$. Excluding the case $N \simeq PSL(2,8)$ and $p=7$, a GAP computation shows that $|spec(M)| \geqslant 6$. By Lemma \ref{lem.spec}, $\omega(G) \geqslant 7$. Finally, if $|M| = 7|N|$, where $N \simeq PSL(2,8)$, then $M/C_M(N) \lesssim Aut(N)$. Since $\pi(Aut(N)) = \pi(N)$, we have $C_M(N) \neq \{1\}$. Thus, $\omega(G) \geqslant 7$.
\end{proof}

The following result gives us a description of all nonabelian simple groups with at most $5$ automorphism orbits.

\begin{thm} (Stroppel, \cite[Theorem 4.5]{S}) \label{S.th}
Let $G$ be a non-abelian simple group in which $\omega(G) \leqslant 5$. Then $G$ is isomorphic to one of the groups $A_5$, $A_6$, $PSL(2,7)$ or $PSL(2,8)$.
\end{thm}

\section{Proofs of the main results}

\begin{thmA}
Let $G$ be a nonsolvable group in which $\omega(G) \leqslant 5$. Then $G$ is isomorphic to one of the following groups $A_5$, $A_6$, $PSL(2,7)$ or $PSL(2,8)$. 
\end{thmA}

\begin{proof}
According to Theorem \ref{S.th}, all simple groups with at most $5$ automorphism orbits are $A_5, A_6, PSL(2,7)$ and $PSL(2,8)$. We need to show that every non simple group $G$ with $\omega(G) \leqslant 5$ is solvable. 

Suppose that $G$ is not simple. Note that, if $G$ is caracteristically simple and $\omega(G) \leqslant 5$, then $G$ is simple (Lemma \ref{ch.simple}). Thus, we may assume that $G$ contains a proper nontrivial characteristic subgroup, say $N$. By Remark \ref{prop.spec}, $\omega(N)$ and $\omega(G/N) \leqslant 4$. By \cite[Theorem 3]{LD}, it suffices to prove that $N$ and $G/N$ cannot be isomorphic to $A_5$. If $N \simeq A_5$, then $\omega(G) \geqslant 7$ by Proposition \ref{A5-prop}. Suppose that $G/N \simeq A_5$. Then $N$ is elementary abelian $p$-group, for some prime $p$. For convenience, the next steps of the proof are numbered.

\begin{itemize}
 \item[(1)] Assume $p\neq 2$. 
\end{itemize}
Since a Sylow 2-subgroup of $G$ is not cyclic, we have an element in $G$ of order $2p$ \cite[p. 225]{G}. Therefore $\omega(G) \geqslant 6$. 

\begin{itemize}
 \item[(2)] Assume $p=2$. 
\end{itemize}
In particular, $|g| \in \{2,4\}$ for any $2$-power element outside of $N$. Note that, if $|g|=4$, then $G$ is AT-group. By Lemma \ref{Z.th}, $G$ is simple, a contradiction. So, we may assume that there exists an involution $g$ outside of $N$. We have, $(gh)^2 \in N$, for any $h \in N$. In particular, $gh = hg$, for any $h \in N$. Therefore $N < C_G(N)$. Since $G/N \simeq A_5$, it follows that $N \subseteq Z(G)$. So $\omega(G) \geqslant 6$. Thus $G$ is solvable, which completes the proof. 
\end{proof}

It is convenient to prove first Theorem B under the hypothesis that $|\pi(G)| > 3$ and then extend the result to the general case. 

\begin{prop} \label{aux-prop}
Let $G$ be a nonsolvable group in which $\omega(G) = 6$. If $|\pi(G)| > 3$, then $G \simeq PSL(3,4)$.
\end{prop}

\begin{proof} 
Assume that $G$ is not characteristically simple. Let $N$ be a proper nontrivial characteristic subgroup of $G$. By Remark \ref{prop.spec}, $N$ and $G/N$ have at most $5$ automorphism orbits. Since $G$ is nonsolvable, we have $N$ or $G/N$  nonsolvable.  

Suppose that $N$ is nonsolvable. By Theorem A, $N$ is isomorphic to one of the groups $A_5,A_6,PSL(2,7)$ or $PSL(2,8)$. By Lemma \ref{simple-lem}, $\omega(G) \geqslant 8$. Thus, we may assume that $G/N$ is nonsolvable. By Theorem A, $G/N$ is isomorphic to one of the groups $A_5,A_6,PSL(2,7)$ or $PSL(2,8).$ Let $\pi(G) = \{2,3,p,q\}$ and $\pi(G/N) = \{2,3,p\}$. Since $\omega(N) \leqslant 3$, it follows that there exists a characteristic elementary abelian  $q$-subgroup $Q$ in $N$ (\cite[Theorem 2]{LD}). Without loss of generality we can assume that $Q=N$. By Schur-Zassenhaus Theorem \cite[p. 221]{G}, there exists a complement for $N$ in $G$ (that is, there exists a subgroup $K$ such that $G = KN$ and $K \cap N = 1$). In particular, $K \simeq G/N$. Since $|Aut(K)|$ and $|N|$ are coprime numbers, it follows that $G$ is the direct product of $N$ and $K$. Arguing as in the proof of Lemma \ref{ch.simple} we deduce that $\omega(G) \geqslant 8$. 

We may assume that $G$ is characteristically simple. By Lemma \ref{ch.simple}, $G$ is simple. Using Kohl's classification \cite{Kohl}, $G \simeq PSL(3,4)$. The result follows.
\end{proof}

\begin{ex} \label{ex.N}
Using GAP we obtained one example of nonsolvable and non simple group $G$ such that $|G| = 960$ and $\omega(G)=6$. Moreover, there exists a normal subgroup $N$ of $G$ such that
$$
G/N \simeq A_5 \ \mbox{e} \ N \simeq C_2 \times C_2 \times C_2 \times C_2.
$$
\end{ex}

\begin{thmB} \label{th.B}
Let $G$ be a nonsolvable group in which $\omega(G) = 6$. Then one of the following holds:
\end{thmB}

\begin{itemize}
 \item[(i)] $G \simeq PSL(3,4)$; 
 \item[(ii)] There exists a characteristic elementary abelian $2$-subgroup $N$ of $G$ such that $G/N \simeq A_5$.
\end{itemize}

\begin{proof} 
By Lemma \ref{ch.simple}, if $G$ is characteristically simple, then $G$ is simple. According to Kohl's classification \cite{Kohl}, $G \simeq PSL(3,4)$. In particular, by Proposition \ref{aux-prop}, if $|\pi(G)| \geqslant 4$, then $G \simeq PSL(3,4)$. So, we may assume that $|\pi(G)| = 3$ and $G$ is not characteristically simple. We need to show that for every non simple  and nonsolvable group $G$ with $\omega(G) = 6$, there exists a proper characteristic subgroup $N$ such that $G/N \simeq A_5$, where $N$ is an elementary abelian $2$-subgroup. 

Let $N$ be a proper characteristic subgroup of $G$. For convenience, the next steps of the proof are numbered.

\begin{itemize}
\item[(1)] Assume that $\omega(N) = 2$.
\end{itemize}

So, $\omega(G/N) = 4$ or $5$ and $N$ is elementary abelian $p$-group, for some prime $p$. According to Theorem A and Example \ref{ex.N}, it is sufficient to consider $G/N$ isomorphic to one of the groups $A_6, PSL(2,7)$ or $PSL(2,8)$. Since the Sylow 2-subgroup of $G/N$ is not cyclic, it follows that the subgroup $N$ is elementary abelian $2$-subgroup \cite[p.\ 225]{G}. Suppose that $G/N \simeq PSL(2,8)$. Arguing as in the proof of Theorem A we deduce that $G$ is $AT$-group, a contradiction. Now, we may assume that $G/N \in \{ A_6, PSL(2,7)\}$. Without loss of generality we can assume that there are elements $a \in G \setminus N$ and $h \in N$ such that $|a| = 2$ and $|ah| = 4$. Then there exist the only one automorphism orbit in which it elements has order $4$, $\{(ah)^{\varphi}  \mid \ \varphi \in Aut(G) \}$. On the other hand, $aN$ has order $2$ and $\omega(G) = 6$. Therefore $G/N$ cannot contains elements of order $4$, a contradiction.    

\begin{itemize}
\item[(2)] Assume that $\omega(N) = 3$.
\end{itemize}

There exists a characteristic subgroup $Q$ of $N$ and of $G$ (\cite[Theorem 2]{LD}). As $G/N$ and $G/Q$ are simple, we have a contradiction.

\begin{itemize}
\item[(3)] Assume that $\omega(N) = 4$ or $5$.
\end{itemize}

In particular, $\omega(G/N) \leqslant 3$. Arguing as in $(2)$ we deduce that $\omega(G/N) = 2$. By Theorem A, $N$ is
simple. Hence $$N \in \{A_5, A_6, PSL(2,7), PSL(2,8)\}.$$ By Proposition \ref{A5-prop}, $\omega(G) \geqslant 7$.    
\end{proof}

\subsection*{Acknowledgment}

The authors wishes to express their thanks to S\'ilvio Sandro for several helpful comments concerning ``GAP''.

\end{document}